\documentclass[10pt,reqno]{amsart}
\usepackage{amsmath} 
\usepackage{amssymb}
\usepackage{dsfont}
\usepackage[dvips,draft,final]{graphics}
\usepackage{amssymb,amsmath}
\usepackage[T1]{fontenc}
\usepackage{fancyhdr}
\usepackage{url}

\usepackage{color}
\usepackage{graphicx}

\def\squarebox#1{\hbox to #1{\hfill\vbox to #1{\vfill}}}

\theoremstyle{plain}

\newtheorem{Thm}{Theorem}

\newtheorem{lem}{Lemma}

\newcommand{\R}{\mathbb{R}}

\def\epsilon{\varepsilon}
\def\phi {\varphi}

\newtheorem{rem}{Remark}

\providecommand{\abs}[1]{\left\lvert#1\right\rvert}
\providecommand{\norm}[1]{\left\lVert#1\right\rVert}
\numberwithin{equation}{section}
\renewcommand{\d}{\textrm{d}}
\renewcommand{\leq}{\leqslant}
\renewcommand{\geq}{\geqslant}
\providecommand{\abs}[1]{\left\lvert#1\right\rvert}
\providecommand{\norm}[1]{\left\lVert#1\right\rVert}

\begin{document}

\title{Stability of the determination of a time-dependent coefficient in parabolic equations}
\author[Mourad Choulli and Yavar Kian]{ Mourad Choulli and Yavar Kian}
\address {Mourad Choulli,
LMAM, UMR 7122, Universit\'e de Lorraine,
Ile du Saulcy, 57045 Metz cedex 1, France}
\email{\url{mourad.choulli@univ-lorraine.fr}}
\address { Yavar Kian,
UMR-7332, Aix Marseille Universit\'e, Centre de Physique Théorique, Campus de Luminy, Case 907
13288 Marseille cedex 9, France}
\email{\url{yavar.kian@univ-amu.fr}}

\maketitle
\begin{abstract} We establish a Lipschitz stability estimate  for the inverse problem consisting in the determination of the coefficient $\sigma(t)$,  appearing in a Dirichlet initial-boundary value problem for the parabolic equation $\partial_tu-\Delta_x u+\sigma(t)f(x)u=0$, from Neumann boundary data.   We extend this result  to the same inverse problem when the previous linear parabolic equation  in changed to the semi-linear parabolic equation $\partial_tu-\Delta_x u=F(t,x,\sigma(t),u(x,t))$.

\medskip
\noindent
{\bf Key words :} parabolic equation, semi-linear parabolic equation, inverse problem, determination of time-depend coefficient, stability estimate.

\medskip
\noindent
{\bf AMS subject classifications :} 35R30.
\end{abstract}

\tableofcontents

\section{Introduction}
Throughout this paper, we assume that $\Omega$ is  a $\mathcal C^3$    bounded domain of  $\R^n$ with $n\geq2$. Let $T>0$ and set \[ Q=\Omega\times (0,T),\quad \Gamma=\partial\Omega,\quad \Sigma=\Gamma\times (0,T).\]
We consider the following initial-boundary
value problem
\begin{equation}\label{eq1} 
\left\{\begin{aligned}
&\partial_tu-\Delta_x u+\sigma(t)f(x)u=0,\quad &(x,t)\in Q,\\  &u(x,0)=h(x), &x\in\Omega,\\ &u(x,t)=g(x,t),&\quad (x,t)\in\Sigma.\end{aligned}\right.\end{equation}

We introduce the following assumptions :
\begin{enumerate}
\item[$\rm(H1)$]$f\in \mathcal C^2(\overline{\Omega})$, $h\in \mathcal C^{2,\alpha}(\overline{\Omega})$, $g\in\mathcal C^{2+\alpha,1+\frac{\alpha}{2}}(\overline{\Sigma})$, for some $0<\alpha<1$, and satisfy the compatibility condition
\[\partial_tg(x,0)-\Delta_xh(x)+\sigma(0)f(x)h(x)=0,\quad x\in\Gamma.\]
\item[$\rm(H2)$] There exists $x_0\in\Gamma$ such that
\[\inf_{t\in[0,T]}|g(x_0,t)f(x_0)|>0.\]
\end{enumerate} 

Under assumption (H1), It is well known that, for $\sigma\in\mathcal C^1[0,T]$,  the initial-boundary value problem \eqref{eq1} admits a unique solution $u=u(\sigma )\in\mathcal C^{2+\alpha,1+\frac{\alpha}{2}}(\overline{Q})$ (see Theorem 5.2 of \cite{LSU}). Moreover, given $M>0$, there exists a constant $C>0$ depending only on data (that is $\Omega$, $T$, $f$, $g$ and $h$) such that $\|\sigma\|_{W^{1,\infty}(0,T)}\leq M$ implies
\begin{equation}\label{0}
\|u(\sigma )\|_{C^{2+\alpha ,1+\alpha /2}(\overline{Q})}\leq C.
\end{equation}

\smallskip
In the present paper we are concerned with the inverse problem consisting in
the determination of  the time dependent coefficient $\sigma(t)$  from Neumann boundary data $\partial_\nu u(\sigma )$ on $\Sigma$, where $\partial_\nu$ is the derivative in the direction of the unit outward
normal vector to $\Gamma$.
 
 \smallskip
We prove the following theorem, where $B(M)$ is the ball of $C^1[0,T]$ centered at $0$ and with radius $M>0$.
 
\begin{Thm}\label{t1} 
Assume that $\rm(H1)$ and $\rm(H2)$ are fulfilled. For $i=1,2$, let $\sigma_i\in B(M)$ and $u_i=u(\sigma_i)$.  Then there exists a constant $C>0$, depending only on data, such that
\begin{equation}\label{th1} 
\norm{\sigma_2-\sigma_1}_{L^\infty(0,T)}\leq C\norm{\partial_t\partial_\nu u_2-\partial_t\partial_\nu u_1}_{L^\infty(\Sigma)}.
\end{equation}
\end{Thm}

Following  \cite{COY}, it is quite natural to extend Theorem \ref{t1} when the linear parabolic equation is changed to a semi-linear parabolic equation. To this end, introduce the following semi-linear initial-boundary value problem :
\begin{equation}\label{eqsem1}
\left\{\begin{aligned}
&\partial_tu-\Delta_x u=F(x,t,\sigma(t),u(x,t)),\quad &(x,t)\in Q,\\  &u(x,0)=h(x), &x\in\Omega,\\ &u(x,t)=g(x,t), &(x,t)\in\Sigma 
\end{aligned}\right.
\end{equation}
and consider the following assumptions
\begin{enumerate}
\item[$\rm(H3)$] $h\in \mathcal C^{2,\alpha}(\overline{\Omega})$, $g\in\mathcal C^{2+\alpha,1+\frac{\alpha}{2}}(\overline{\Sigma})$, for some $0<\alpha<1$, and satisfy the compatibility condition
\[\partial_tg(x,0)-\Delta_xh(x)=F(0,x,\sigma(0),h(x)),\quad x\in\Gamma.\]
\item[$\rm(H4)$] $F\in\mathcal C^1(\overline{\Omega}_x\times\R_t\times\R_\sigma\times\R_u)$ is such that $\partial_u F$ and $\partial_\sigma F$ are $\mathcal C^1$, $F$ and $\partial_\sigma F$ are  $\mathcal C^2$ with respect to $x$ and $u$. 
\item[$\rm(H5)$] There exist $M>0$ and  $x_0\in\Gamma$ such that
\[\inf_{t\in[0,T],\sigma\in[-M,M]}|\partial_\sigma F(x_0,t,\sigma,g(x_0,t))|>0.\]
\item[$\rm(H6)$] There exist two non negative constants $c$ and $d$ such that
\[uF(x,t,\sigma(t),u)\leq cu^2+d,\quad t\in[0,T],\ \ x\in\overline{\Omega},\ \ u\in\R.\]
\end{enumerate} 

Under the above mentioned conditions, for any $\sigma\in\mathcal C^1[0,T]$, the initial-boundary value problem \eqref{eqsem1} admits a unique solution $u=u(\sigma )\in\mathcal C^{2+\alpha,1+\frac{\alpha}{2}}(\overline{Q})$ (see Theorem 6.1 in \cite{LSU}) and, given $M>0$, there exists a constant $C>0$ depending only on data (that is $\Omega$, $T$, $F$, $g$ and $h$) such that $\|\sigma\|_{W^{1,\infty}(0,T)}\leq M$ implies
\begin{equation}\label{0}
\|u(\sigma )\|_{C^{2+\alpha ,1+\alpha /2}(\overline{Q})}\leq C.
\end{equation}

We have the following extension of Theorem \ref{t1}.

\begin{Thm}\label{t2} 
Assume that $\rm(H3)$, $\rm(H4)$, $\rm(H5)$ and $\rm(H6)$ are fulfilled. For $i=1,2$, let $\sigma_i\in B(M)$ and $u_i=u(\sigma_i)$.  Then there exists a constant $C>0$, depending only on data, such that
\begin{equation}\label{th2b} 
\norm{\sigma_2-\sigma_1}_{L^\infty(0,T)}\leq C\norm{\partial_t\partial_\nu u_2-\partial_t\partial_\nu u_1}_{L^\infty(\Sigma)}.
\end{equation}
\end{Thm}

\begin{rem} Let us observe that we can generalize the results in Theorems \ref{t1} and \ref{t2} as follows:

\smallskip
i) In \eqref{eq1}, we can replace $\sigma (t)f(x)$ by $\sum_{k=1}^p\sigma _k(t)f_k(x)$, where $f_k$, $1\leq k\leq p$, are known. Assume that $(H1)$ is satisfied, with $f=f_k$ for each $k$, where the compatibility condition is changed to 
\[\partial_tg(x,0)-\Delta_xh(x)+\sum_{k=1}^p\sigma_k(0)f_k(x)h(x)=0,\quad x\in\Gamma.\]
Therefore, to each $(\sigma _1,\ldots ,\sigma _p)\in \mathcal{C}[0,T]^p$ corresponds a unique solution $u=u(\sigma _1,\ldots ,\sigma _p)\in C^{2+\alpha ,1+\alpha /2}(\overline{Q})$ and $\max\{ \|\sigma _k\|_{W^{1,\infty}(0,T)};\; 1\leq k\leq p \}\leq M$ implies
\[
\|u(\sigma _1,\ldots ,\sigma _p)\|_{C^{2+\alpha ,1+\alpha /2}(\overline{Q})}\leq C,
\]
for some positive constant $C$ depending only on data.

\smallskip
Following the proof of Theorem \ref{t1}, we prove that, under the following conditions : there exists $x_1,\ldots ,x_p \in \Gamma$ such that the matrix $M(t) = (f_k(x_l)g(x_l,t))$ is invertible  for any $t\in [0,T]$,
\[
\max_{1\leq k\leq p}\|\sigma _k^1-\sigma _k^2\|_{L^\infty (0,T)}\leq C\norm{\partial_t\partial_\nu u_2-\partial_t\partial_\nu u_1}_{L^\infty(\Sigma)},
\]
if $\sigma _k^j\in B(M)$, $1\leq k\leq p$ and $j=1,2$. Here $C$ is a constant that can depend only on data and $u_j=u(\sigma _1^j,\ldots ,\sigma _p^j)$, $j=1,2$.

\smallskip
ii)  We can replace the semi-linear parabolic equation in \eqref{eqsem1} by a semi-linear integro-differential equation. In other words, $F$ can be changed to
\[
F_1(x,t,\sigma (t), u(x,t))+\int_0^tF_2(x,s,\sigma (t-s),u(x,s))ds.
\]
Under appropriate assumptions on $F_1$ and $F_2$, one can establish that Theorem \ref{t2} is still valid in the present case.

\smallskip
ii) Both in \eqref{eq1} and \eqref{eqsem1}, the Laplace operator can be replaced by a second order elliptic operator in divergence form :
\[
E= \nabla \cdot A(x)\nabla +B(x)\cdot \nabla ,
\]
where $A(x)=(a_{ij}(x))$ is a symmetric matrix with coefficients in $C^{1+\alpha}(\overline{\Omega})$, $B(x)=(b_i(x))$ is a vector with components in $C^\alpha (\overline{\Omega})$ and the following ellipticity condition holds
\[A(x)\xi\cdot\xi\geq\lambda\abs{\xi}^2,\quad \xi\in\R^n,\;  x\in\overline{\Omega}.\]
Actually, the normal derivative associated to $E$ is the boundary operator $\partial _{\nu _E}=\nu (x)\cdot A(x)\nabla$.
\end{rem}

To our knowledge, there are only few results concerning the determination of a time-dependent coefficient in an initial-boundary value problem for a parabolic equation from a single measurement. The determination of a source  term of the form $f(t)\chi_D(x)$, where $\chi_D$ the characteristic function of the known  subdomain $D$, was considered by J. R. Canon and S. P. Esteva. They established in \cite{CE86-1} a logarithmic stability estimate in 1D case in a half line when the overdetermined data is the trace at the end point. A similar inverse problem problem in 3D case was studied by these authors in \cite{CE86-2}, where they obtained a Lipschitz stability estimate in weighted spaces of continuous functions. The case of a non local measurement was considered by J. R. Canon and Y. Lin  in \cite{CL88} and \cite{CL90}, where they proved existence and uniqueness for both quasilinear and semi-linear parabolic equations. The determination of a time dependent coefficient in an abstract integrodifferential equation was studied by the first author in \cite{Ch91-1}. He proved existence, uniqueness and Lipschitz stability estimate, extending earlier results by \cite{Ch91-2}, \cite{LS87}, \cite{LS88}, \cite{PO85-1} and \cite{PO85-2}. In \cite{CY06}, the first author and M. Yamamoto obtained a stability result, in a restricted class,  for the inverse problem of determining a source term $f(x,t)$, appearing in a Dirichlet initial-boundary value problem for the heat equation, from Neumann boundary data. In a recent work, the first author and M. Yamamoto \cite{CY11} considered the inverse problem of  finding a control parameter $p(t)$ that reach a desired temperature $h(t)$ along a curve $\gamma (t)$ for a parabolic semi-linear equation with homogeneous Neumann boundary data and they established existence, uniqueness as well as Lipschitz stability. Using geometric optic  solutions, the first author \cite{Ch09} proved uniqueness as well as stability for the inverse problem of determining a general time dependent coefficient of order zero for parabolic equations  from Dirichlet to Neumann map.
In \cite{E07} and \cite{E08}, G. Eskin considered the same inverse problem for hyperbolic and the Schr\"odinger equations with time-dependent electric and magnetic potential and he established uniqueness by gauge invariance.  Recently, R. Salazar \cite{Sa} extended the result of \cite{E07} and obtained a stability result for compactly supported coefficients.

\smallskip
We would like to mention that  the determination of space dependent coefficient $f(x)$, in the source term $\sigma(t)f(x)$, from Neumann boundary data was already considered by the first author and  M. Yamamoto \cite{CY06}.  But, it seems that  our paper is the first work where one treats the determination of a time dependent coefficient, appearing in a parabolic initial-boundary value problem, from Neumann boundary data.

\smallskip
This paper is organized as follows.  In section 2 we come back to the construction of the Neumann fundamental solution by \cite{It} and establish time-differentiability of some potential-type functions, necessary for proving Theorems 1 and 2. Section 3 is devoted to the proof of Theorems 1 and 2.

\section{Time-differentiability of potential-type functions}

In this section, we establish time-differentiability of some potential-type functions, needed in the proof of our stability estimates. In our analysis we follow the construction of the fundamental solution by S. It\^o \cite{It}.

\smallskip
First of all, we recall the definition of fundamental solution associate to the heat equation plus a time-dependent coefficient  of order zero, in the case of Neumann boundary condition. Consider the initial-boundary value problem
\begin{equation}\label{eq2}
\left\{
\begin{aligned}
&\partial_tu =\Delta_x u+q(x,t)u, \quad &(x,t)\in \Omega\times (s,t_0),\\  &\lim_{t\to s}u(x,t)=u_0(x), &x\in\Omega ,\\ &\partial_\nu u(x,t)=0, &(x,t)\in\Gamma\times (s,t_0).
\end{aligned}\right.
\end{equation}
Here $s_0<t_0$ are fixed, $s\in (s_0,t_0)$, $u_0$ and $q(x,t)$ are continuous  respectively in $\overline{\Omega}$ and in $\overline{\Omega}\times[s,t_0]$. Let $U(x,t;y,s)$ be a continuous function in the domain  $s_0<s<t<t_0$, $x\in\overline{\Omega}$, $y\in\overline{\Omega}$. We recall that $U$ is the fundamental solution of \eqref{eq2} if for any $u_0\in\mathcal C(\overline{\Omega})$, 
\[u(x,t)=\int_\Omega U(x,t;y,s)u_0(y)\d y\]
is the solution of \eqref{eq2}. We refer to \cite{It} for the existence and uniqueness of this fundamental solution.

\smallskip
We start with time-differentiability of volume potential-type functions\footnote{Recall that if $\varphi=\varphi (x,t)$ is a continuous function then the corresponding volume potential is given by \[\psi (x,t)=\int_s^t \int_\Omega U(x,t;y,\tau )\varphi (y,\tau )\d y\d \tau .\]}.

\begin{lem}\label{l1}
Fix $s\in(s_0,t_0)$. Let $f\in\mathcal C(\overline{\Omega}\times[s,t_0])$ be $\mathcal C^2$ with respect to $x$, $q\in\mathcal C^1( \overline{\Omega}\times[s,t_0])$ and define, for $(x,t)\in\overline{\Omega}\times(s,t_0)$,
\[f^1(x,t;\tau)=\int_{\Omega}U(x,t;y,\tau)f(y,\tau)\d y,\quad t>\tau>s.\]
Then, $f^1$ admits a derivative with respect to $t$ and 
\begin{equation}\label{l1a}\begin{aligned}\frac{\partial f^1}{\partial t}(x,t;\tau)=&\int_\Omega U(x,t;y,\tau)(\Delta_y+q(x,\tau))f(y,\tau)\d y\\
&+\int_\tau^t\int_\Omega\int_\Omega U(x,t;z,\tau')\partial _tq(z,\tau')U(z,\tau';y,\tau)f(y,\tau)\d z \d y \d \tau'.\end{aligned}\end{equation}
Moreover, $F$ given by
\[F(x,t)=\int_s^t f^1(x,t;\tau)\d \tau,\; (x,t)\in\Gamma \times(s_0,t_0),\]
possesses a derivative with respect to $t$, 
\begin{equation}\label{l1b}\frac{\partial F}{\partial t}(x,t)=f(x,t)+\int_s^t\frac{\partial f^1}{\partial t}(x,t;\tau)\d \tau\end{equation}
and
\begin{equation}\label{l1c}\abs{\int_s^t\frac{\partial f^1}{\partial t}(x,t;\tau)\d \tau}\leq C\int_s^t\norm{f(.,\tau)}_{\mathcal \mathcal C_x^2\left(\overline{\Omega}\right)}\d\tau .\end{equation}
\end{lem}

\begin{proof}
We have only to prove \eqref{l1a} and \eqref{l1c}, because \eqref{l1b} follows immediately from \eqref{l1a}.%MTheorem 7.1 of \cite{It}.

\smallskip
Let then $u_0\in\mathcal C^2(\overline{\Omega})$ and consider the function 
\[u(x,t)=\int_\Omega U(x,t;y,s)u_0(y)\d y,\quad   x\in\overline{\Omega},\; s<t<t_0.\]
We show that $u$ admits a derivative with respect to $t$ and
\begin{equation}
\label{l1a1}\begin{aligned}
\partial_tu(x,t) &=\partial_t\left(\int_\Omega U(x,t;y,s)u_0(y)\d y\right)
\\
&=\int_\Omega U(x,t;y,s)(\Delta_y+q(x,s))u_0(y)\d y
\\
&\hskip 1.5 cm -\int_s^t\int_\Omega\int_\Omega U(x,t;z,\tau)q_t(z,\tau)U(z,\tau;y,s)u_0(y)\d z \d y \d \tau.
\end{aligned}
\end{equation}

We need to consider first the case  $u_0=w_0\in\mathcal C^\infty(\overline{\Omega})$. Set
\[
w(x,t)=\int_\Omega U(x,t;z,s)w_0(y)\d y,\quad x\in\overline{\Omega},\; s<t<t_0.
\]
Clearly, $w(x,t)$ is the solution of the following initial-boundary value problem
\[
\left\{
\begin{aligned}
&\partial_tw-\Delta_x w-q(x,t)w=0, \quad &(x,t)\in \Omega\times (s,t_0),\\  
&\lim_{t\to s}w(x,t)=w_0(x), &x\in\Omega,\\ 
&\partial_\nu w(x,t)=0, &(x,t)\in\Gamma\times(s,t_0)
\end{aligned}
\right.
\]
and  $w_1=\partial_tw$ satisfies
\[\left\{
\begin{aligned}
&\partial_tw_1-\Delta_x w_1-q(x,t)w_1=-\partial_tqw, \quad  & (x,t)\in \Omega\times (s,t_0),\\  
&\lim_{t\to s}w_1(x,t)=(\Delta_x+q(x,s))w_0(x), &x\in\Omega,\\ 
&\partial_\nu w_1(x,t)=0, &(x,t)\in\Gamma\times(s,t_0).
\end{aligned}
\right.\]
Therefore, \eqref{l1a1}, with $w$ in place of $u$, is a consequence of Theorem 9.1 of \cite{It}. 

\smallskip
Next, let $(w_0^n)_n$ be a sequence in $\mathcal C^\infty(\overline{\Omega})$  converging to $u_0$  in $C^2(\overline{\Omega})$
and $v(x,t)$ given by
\[\begin{aligned}v(x,t)=&\int_\Omega U(x,t;y,s)(\Delta_x+q(x,s))u_0(y)\d y\\
\ &-\int_s^t\int_\Omega\int_\Omega U(x,t;z,\tau)\partial_tq(z,\tau)U(z,\tau;y,s)u_0(y)\d z \d y \d \tau.\end{aligned}\]
Consider $(w_n)_n$, the sequence of functions, defined by
\[w_n(x,t)=\int_\Omega U(x,t;z,s)w^n_0(y)\d y.\]
We proved that, for any $n\in\mathbb N$, 
\begin{equation}\label{l1a3}
\begin{aligned}
\partial_tw_n(x,t)
&=\int_\Omega U(x,t;y,s)(\Delta_y+q(x,s))w^n_0(y)\d y
\\
&\hskip 1.5 cm -\int_s^t\int_\Omega\int_\Omega U(x,t;z,\tau)\partial_tq(z,\tau)U(z,\tau;y,s)w^n_0(y)\d z \d y \d \tau.
\end{aligned}
\end{equation}
From the proof of Theorem 7.1 of \cite{It}, 
\begin{equation}\label{l1a4}\int_\Omega\abs{U(x,t;y,s)}\d y\leq Ce^{C(t-s)},\; (x,t)\in\overline{\Omega}\times(s,t_0).\end{equation}
Therefore, we can pass to the limit, as $n$ goes to infinity, in  \eqref{l1a3}. We deduce that $\partial _tw_n$ converges to $v$ in $\mathcal C(\overline{\Omega}\times[s,t_0])$. But, $w_n$ converges to $u$ in $\mathcal C(\overline{\Omega}\times[s,t_0])$. Hence $u$ admits a derivative with respect to $t$ and $\partial_tu=v$. That is we proved \eqref{l1a1} and  consequently \eqref{l1a} holds  true. Finally,  we note that \eqref{l1c} is deduced easily from \eqref{l1a4}.
\end{proof}

Next, we consider time-differentiability a single layer potential-type function\footnote{The single-layer potential corresponding to a continuous function $\varphi=\varphi (x,t)$ is   given by \[\psi (x,t)=\int_s^t \int_\Gamma U(x,t;y,\tau )\varphi (y,\tau )\d \sigma (y)\d \tau .\]}.

\begin{lem}\label{l2}
Fix $s\in(s_0,t_0)$. Let $f\in\mathcal C(\Gamma \times [s,t_0])$ be $\mathcal C^1$ with respect to $t\in[s,t_0]$ with $f(x,s)=0$. Define, for $(x,t)\in \Gamma\times (s,t_0)$,
\[f^1(x,t;\tau)=\int_{\Gamma}U(x,t;y,\tau)f(y,\tau)\d \sigma(y),\quad t>\tau>s.\]
Then\[F(x,t)=\int_s^t f^1(x,t;\tau)\d \tau\]
is well defined, admits a derivative with respect to $t$ and we have
\begin{equation}\label{l2a}\norm{\frac{\partial F}{\partial t}}_{L^\infty(\Gamma \times (s,t_0))}\leq C\norm{\partial_tf}_{L^\infty(\Gamma \times (s,t_0) )}.\end{equation}
\end{lem}

Contrary to Lemma \ref{l1}, for Lemma \ref{l2} we cannot use directly the general properties of the fundamental solutions developed in \cite{It}.  We need to come back to the construction of the fundamental solution of \eqref{eq2} introduced by \cite{It}. First, consider the heat equation $\partial_tu =\Delta_xu$ in the half space $\Omega_1=\{x=(x_1,\ldots,x_n);\; x_1>0\}$ in $\R^n$ with the boundary condition
$\partial_{x_1}u=0$ on $\Gamma_1= \{x=(0,x_2,\ldots,x_n);\; (x_2,\ldots,x_n)\in\R^{n-1}\}$. For any $y=(y_1,y_2,\ldots,y_n)$, we define $\overline{y}$ by $\overline{y}=(-y_1,y_2,\ldots,y_n)$. Let 
\[G(x,t)=\frac{1}{(4\pi t)^{\frac{n}{2}}}e^{-\frac{\abs{x}^2}{t}}\]
denotes the Gaussian kernel and set 
\[G_1(x,t;y)=G(x-y,t)+G(x-\overline{y},t).\]
Then, the fundamental solution $U_0(x,t;y,s)$ of
\begin{equation}\label{eq3}
\left\{\begin{aligned}& \partial_tu=\Delta_x u, \quad &(x,t)\in \Omega_1\times (s,t_0),\\  &\lim_{t\to s}u(x,t)=u_0(x),&x\in\Omega_1,\\ &\partial_\nu u(x,t)=0,&(x,t)\in\Gamma_1\times (s,t_0)\end{aligned}\right.\end{equation}
is given by
\[U_0(x,t;y,s)=G_1(x,t-s;y).\]
In order to construct the fundamental solution in the case of an arbitrary domain $\Omega$, Itô introduced the following local coordinate system around each point $z\in\Gamma$.
\begin{lem}\label{l3}
\emph{(Lemma 6.1 and its corollary, Chapter 6 of \cite{It})} For every point $z\in\Gamma$, there exist a coordinate neighborhood $W_z$ of $z$ and a coordinate system $(x^*_1,\ldots,x^*_n)$ satisfying the following conditions:\\
1) the coordinate transformation between the coordinate system $(x^*_1,\ldots,x^*_n)$ and the original coordinate system in $W_z$ is of class $C^2$ and the partial derivatives of the second order of the transformation functions are H\"older continuous ;\\
2) $\Gamma\cap W_z$ is represented by the equation $x^*_1=0$ and $\Omega\cap W_z$  is represented by $x^*_1>0$ ;\\
3) let $\mathcal L$ be the diffeomorphism from $W_z$ to $\mathcal L(W_z)$ defined by
\[\begin{array}{rccl} \mathcal L: & W_z& \to & \mathcal L(W_z) \\
 \ \\ & x & \mapsto &  (x^*_1(x),\ldots,x^*_n(x)).\end{array}\]
Then, for any $u\in\mathcal C^1(\overline{\Omega})$ we have 
\[\partial_\nu u(\xi)=-\partial_{x_1}(u\circ\mathcal L^{-1})(x),\quad\xi\in\Gamma\cap W_z\ \ \textrm{and }\ x=\mathcal L(\xi).\]
\end{lem}

From now on, for any $z\in\Gamma$,  we view coordinate system $(x^*_1,\ldots,x^*_n)$ as a rectangular coordinate system. Moreover, using the local coordinate system of Lemma \ref{l3}, for any $y=(y_1,y_2,\ldots,y_n)\in \mathcal L(W_z)$, we define $\overline{y}=(-y_1,y_2,\ldots,y_n)$ and, without loss of generality, we  assume that, for any $y\in \mathcal L(W_z)$, we have $\overline{y}\in \mathcal L(W_z)$. For any interior point $z$ of $\Omega$, we fix an arbitrary local coordinate system and a coordinate neighborhood $W_z$ contained in $\Omega$. For any $z\in\overline{\Omega}$ and $\delta>0$, we set $W(z,\delta)=\{x:\ \abs{x-z}^2<\delta\}$ and $\delta_z>0$ such that, for any $z\in\overline{\Omega}$ we have $\overline{W(z,\delta_z)}\subset W_z$.

\smallskip
Recall the following partition of unity lemma.
 \begin{lem}\label{l4}\emph{(Lemma 7.1, Chapter 7 of \cite{It})} There exist a finite subset $\{ z_1,\ldots,z_m\}$ of $\overline{\Omega}$ and a finite sequence of functions $\{\omega_1,\ldots,\omega_m\}$ with the following properties:\\
 1) $\textrm{supp}\; \omega_l\subset W(z_l,\delta_{z_l})$, $l=1,\ldots,m$, and each $\omega_l$ is of class $\mathcal C^3$ with respect to the local coordinates in $W_{z_l}$  ;\\
 2) $\{\omega_l(x)^2;\; l=1,\ldots,m\}$ forms a partition of unity in $\overline{\Omega}$ ;\\
 3)$\partial_\nu\omega_l(\xi)=0,$ $l=1,\ldots,m$, $\xi\in\Gamma$.
\end{lem}

Let $\{ z_1,\ldots,z_m\}$  be the finite subset of $\overline{\Omega}$, introduced in the previous lemma. For any $k\in\{1,\ldots,m\}$, let $\mathcal L_k$  denotes the diffeomorphism from $W_{z_k}$ to $\mathcal L_k(W_{z_k})$ defined by
\[\begin{array}{rccl} \mathcal L_k: & W_{z_k}& \to & \mathcal L_k(W_{z_k}) \\
 \ \\ & x & \mapsto &  (x^*_1(x),\ldots,x^*_n(x)),\end{array}\]
where $(x^*_1,\ldots,x^*_n)$ is the local coordinate system of Lemma \ref{l3} defined in $W_{z_k}$. For any $k\in\{1,\ldots,m\}$, the differential operator
\[\partial_t-\Delta_x-q(x,t)\]
becomes, in terms of local coordinate  system $x^*=(x^*_1,\ldots,x^*_n)$,
\[L^k_{t,x^*}=\partial_t-\frac{1}{\sqrt{a_k(x^*)}}\sum_{i,j=1}^n\partial_{x^*_i}\left(\sqrt{a_k(x^*)}a^{ij}_k(x^*)\partial_{x^*_j}\cdot\right)-q_k(x^*,t)\]
in $ \mathcal L_k(W_{z_k})\times(s_0,t_0)$. Here  $q_k(x^*,t)$ is H\"older continuous on $\mathcal L_k(W_{z_k})\times(s_0,t_0)$ and $(a^{ij}(x^*))$ is the contravariant tensor of degree 2  defined by
\[\left(a^{ij}_k(x^*)\right)=\left(J_{\mathcal L_k}(\mathcal L_k^{-1}(x^*))\right)^{T}\left(J_{\mathcal L_k}(\mathcal L_k^{-1}(x^*))\right),\]
with
\[J_{\mathcal L_k}(x)=\left(\frac{\partial x^*_j(x)}{\partial x_i}\right).\]

According to the construction of \cite{It}  given in Chapter 6  (see pages 42 to 45 of \cite{It}), $\left(a^{ij}_k(x^*)\right)$ is of class $\mathcal C^2$ in $\mathcal L_k(\overline{\Omega}\cap W_{z_k})$ and it is  a positive definite symmetric matrix at every point $x^*\in \mathcal L_k(W_{z_k})$. We set $(a^k_{ij}(x^*))=(a^{ij}_k(x^*))^{-1}$ and $a_k(x^*)=\rm{det}(a^k_{ij}(x^*))$. Consider the volume element $\d x^*=\sqrt{a_k(x^*)}\d x^*_1\ldots\d x^*_n$ on $\mathcal L_k(W_{z_k})$ and $\d x'=\sqrt{a(0,x')}\d x^*_2\ldots\d x^*_n$ on $\mathcal L_k(W_{z_k}\cap\Gamma)$ with $x'=(x^*_2,\ldots,x^*_n)$. Note that, by the construction of S. It\^o \cite{It} (see page 45), for any $k=1,\ldots,m$, we have
\begin{equation}\label{met1}a_k^{ij}(\mathcal L_k(x))=a_k^{ij}(\overline{\mathcal L_k(x)}),\quad x\in\overline{\Omega}\cap W_{z_k},\ \ \textrm{for }i=j=1\textrm{ or }i,j=2,\ldots,n,\end{equation}
\begin{equation}\label{met2}a_k^{1j}(\mathcal L_k(x))=a_k^{j1}(\mathcal L_k(x))=-a_k^{1j}(\overline{\mathcal L_k(x)}),\quad x\in\overline{\Omega}\cap W_{z_k},\ \ \textrm{for }j=1,\ldots,n\end{equation}
and
\begin{equation}\label{met3}a_k^{1j}(\mathcal L_k(\xi))=a_k^{j1}(\mathcal L_k(\xi))=\delta_{j1},\quad\xi\in\Gamma\cap W_{z_k},\ \ j=1,\ldots,n,\end{equation}
where $\delta_{j1}$ denotes the kronecker's symbol.
For any $k\in\{1,\ldots,m\}$, let $G_k(x,t;y)$ be defined, in the region
\[D_k=\{(x,t,y);\;  x,\ y\in \mathcal L_k(W_{z_k}),\; 0<t<t_0-s\},\]
by
\[G_k(x,t;y)=\frac{1}{(4\pi t)^{\frac{n}{2}}}e^{-\sum_{i,j=1}^n\frac{a^k_{ij}(y)(x_i-y_i)(x_j-y_j)}{4t}}.\]

Next , define $H_{z_k}(x,t;y)=G_k(\mathcal L_k(x),t;\mathcal L_k(y))$, for $k\in\{1,\ldots,m\}$ and $z_k\in\Omega$ ;   $H_{z_k}(x,t;y)=G_k(\mathcal L_k(x),t;\mathcal L_k(y))+G_k(\mathcal L_k(x),t;\overline{\mathcal L_k(y)})$, for $k\in\{1,\ldots,m\}$,  $z_k\in\Gamma$, $x\in W_{z_k}$ and $y\in W_{z_k}$ ; $H_{z_k}(x,t;y)=0$ if $x\notin W_{z_k}$ or $y\notin W_{z_k}$. Consider also  $H(x,t;y)$, defined in the region
\[D=\{(x,t,y);\;  x\in\overline{\Omega},\ y\in\overline{\Omega},\; 0<t<t_0-s\},\]
as follows
\[H(x,t;y)=\sum_{l=1}^m\omega_l(x)H_{z_l}(x,t;y)\omega_l(y).\]
As in  Lemma 7.2 of \cite{It}, we define successively:
\[J_0(x,t;y,s)=(\partial_t-\Delta_x-q(x,t))(H(x,t-s;y)),\]
\[J_k(x,t;y,s)=\int_s^t\int_\Omega J_0(x,t;z,\tau)J_{k-1}(z,\tau;y,s)\d z \d\tau,\]
\[K(x,t;y,s)=\sum_{k=0}^{+\infty}J_k(x,t;y,s).\]
Then, following  \cite{It} (see page 53), the fundamental solution of \eqref{eq2} is given by
\begin{equation}\label{it}U(x,t;s,y)=H(x,t-s;y)+\int_s^t\int_\Omega H(x,t-\tau;z)K(z,\tau;y,s)\d z \d\tau.\end{equation}

We are now able to prove Lemma \ref{l2} with the help of representation \eqref{it}, the properties of $H(x,t;y)$ and $K(x,t;y,s)$.

\smallskip
{\it Proof of Lemma \ref{l2}.} Without loss of generality, we assume that $s=0$. Set 
\[F_1(x,t)=\int_0^t\int_{\Gamma}H(x,t-s;y)f(y,s)\d\sigma(y) \d s,\]
\[F_2(x,t)=\int_0^t\int_{\Gamma}\int_s^t\int_\Omega H(x,t-\tau ;z)K(z,\tau;y,s)f(y,s)\d z \d\tau \d\sigma(y) \d s.\]
According to representation \eqref{it}, one needs to show that $F_1$ and $F_2$ admit a derivative with respect to $t$ and 
\begin{equation}\label{l2a}\abs{\partial_tF_1(x,t)}+\abs{\partial_tF_2(x,t)}\leq C\norm{\partial_t f}_{L^\infty(\Gamma \times (0,t_0))}\end{equation}
for $(x,t)\in\Gamma\times(0,t_0)$.
We start by considering $F_1$. Applying a simple substitution, we obtain
\begin{equation}\label{l2b}F_1(x,t)=\int_0^t\int_{\Gamma}H(x,s;y)f(y,t-s)\d\sigma(y) \d s.\end{equation}

Next, for $x\in\Gamma$, there exist $l_1,\ldots,l_r\subset\{1,\ldots,m\}$ such that $x\in \textrm{supp}\, \omega_l$ for $l\in\{l_1,\ldots,l_r\}$ and $x\notin \textrm{supp} \, \omega_l$ for $l\notin\{l_1,\ldots,l_r\}$. Moreover, since $x\in\Gamma$, we have $z_{l_1},\ldots,z_{l_r}\in\Gamma$. Then, from the construction of $H(x,t;y)$, we obtain
\[\begin{aligned}\int_{\Gamma}H(x,s;y)\d\sigma(y)&=\int_{\Gamma}\sum_{k=1}^r\omega_{l_k}(x)H_{z_{l_k}}(x,s;y)\omega_{l_k}(y)\d\sigma(y)\\
\ &=2\sum_{k=1}^r\int_{\R^{n-1}}\chi_{l_k}(0,x')\frac{1}{(4\pi s)^{\frac{n}{2}}}e^{-\sum_{i,j=1}^n\frac{a^{l_k}_{ij}(0,y')(x'_i-y'_i)(x'_j-y'_j)}{4s}}\chi_{l_k}(0,y')\sqrt{a_{l_k}(0,y')}\d y'\end{aligned}\]
with, for $l\in\{1,\ldots,m\}$, $\chi_l\in\mathcal C_0^3\left(\mathcal L_l(\textrm{supp}\, \omega_l)\right)$ such that $\chi_l(x)=\omega_l(\mathcal L_l^{-1}(x))$ and with $(x_1',\ldots, x_n')=(0,x')$, $(y_1',\ldots, y_n')=(0,y')$.
 Using the substitution $y'\rightarrow z'=\frac{x'-y'}{\sqrt{s}}$, we derive
\begin{equation}\label{l2c}\begin{aligned}&\int_{\Gamma}H(x,s;y)\d\sigma(y)\\
&\leq C\sum_{k=1}^r\int_{\R^{n-1}}\chi_{l_k}(0,x')\frac{1}{\sqrt{s}}e^{-\sum_{i,j=1}^n a^{l_k}_{ij}(0,x'-\sqrt{s}z')z'_i z'_j}\chi_{l_k}(0,x'-\sqrt{s}z')\sqrt{a_{l_k}(0,x'-\sqrt{s}z')}\d z'.\end{aligned}\end{equation}
Therefore,
\[\int_{\Gamma}\abs{H(x,s;y)}\d\sigma(y)\leq \frac{C}{\sqrt{s}}\int_{\R^{n-1}}e^{-a_0\abs{z'}^2}dz'\leq \frac{C}{\sqrt{s}},\]
where $a_0>0$ is a constant.
From this estimate, we deduce that
 \[\int_{\Gamma}\abs{H(x,s;y)f(y,t-s)}\d\sigma(y)\leq C\frac{\norm{f}_{L^\infty(\Gamma\times (0,t_0))}}{\sqrt{s}}\]
and
\[\abs{\partial_t\left(\int_{\Gamma}H(x,s;y)f(y,t-s)\d\sigma(y)\right)}\leq C\frac{\norm{\partial_t f}_{L^\infty(\Gamma\times (0,t_0))}}{\sqrt{s}}.\]
Thus, $F_1$ admits a derivative with respect to $t$,
\[\partial_t F_1(x,t)=\int_{\Gamma}H(x,t;y)f(y,0)\d\sigma(y)+\int_0^t\int_{\Gamma}H(x,s;y)\partial_tf(y,t-s)\d\sigma(y)\d s\]
and, since $f(y,0)=0$ for $y\in\Gamma$, we obtain
\begin{equation}\label{l2dd}\abs{\partial_tF_1(x,t)}\leq C\norm{\partial_tf}_{L^\infty(\Gamma\times (0,t_0))},\quad (x,t)\in\Gamma\times (0,t_0).\end{equation}

Let us  now consider $F_2$. We want to show that $\partial _tF_2$ exists and the following estimate holds:
\begin{equation}\label{l2g}\abs{\partial_tF_2(x,t)}\leq C\norm{\partial_tf}_{L^\infty(\Gamma\times (0,t_0) )},\quad (x,t)\in\Gamma\times (0,t_0).\end{equation}
For this purpose, using the local coordinate system, it suffices to prove 
\begin{equation}\label{l2gg}\abs{\partial_tF_2(\mathcal L_l^{-1}(0,x'),t)}\leq C_l\norm{\partial_tf}_{L^\infty((\Gamma\times (0,t_0) )},\quad ((0,x'),t)\in \mathcal L_l(\Gamma\cap W_{z_l})\times(0,t_0),\ \ l\in\{1,\ldots,m\}.\end{equation}
From now on we set  $x=\mathcal L_l^{-1}(0,x')$ with $(0,x')\in\mathcal L_l(\Gamma\cap W_{z_l})\subset\{0\}\times\R^{n-1}$ and we will show \eqref{l2gg}.
First, note that
\[\begin{aligned}J_0(z,\tau;s,y)=&(\partial_\tau-\Delta_z-q(z,t))H(z,\tau-s;y)\\
=&\sum_{l=1}^m\omega_l\left(\mathcal L_l^{-1}(z^*)\right)L^l_{\tau,z^*}H_{z_l}(\mathcal L_l^{-1}(z^*),\tau -s;y)\omega_l(y)\\
\ &+\sum_{l=1}^m[L^l_{\tau,z^*},\omega_l(\mathcal L_l^{-1}(z^*))]H_{z_l}(\mathcal L_l^{-1}(z^*),\tau -s;y)\omega_l(y).\end{aligned}\]
According to the results in Chapter 4 of \cite{It} (pages 26 and 27), using the local coordinate system, we obtain
\[\begin{aligned}L^l_{\tau,z^*}H_{z_l}(\mathcal L_l^{-1}(z^*),\tau -s;\mathcal L_l^{-1}(y^*))=&\sum_{i,j=1}^n(a_l^{ij}(z^*)-a_l^{ij}(y^*))\frac{\partial^2 H_{z_l}}{\partial_{z^*_i}\partial_{z^*_j}}(\mathcal L_l^{-1}(z^*),\tau -s;\mathcal L_l^{-1}(y^*))\\
\ &+[B_l(z^*,y^*,\partial_{z^*})+q_l(z^*,t)]H_{z_l}(\mathcal L_l^{-1}(z^*),\tau -s;\mathcal L_l^{-1}(y^*)),\end{aligned}\]
where $B_l(z^*,y^*,\partial_{z^*})$ is a differential operator of order $\leq 1$ in $z^*$ with continuous coefficients in $z^*,y^*\in\mathcal L_l(\textrm{supp}\, \omega_l)$. In view of the results in Chapter 4 of \cite{It} (see pages 26 and 27), combining \eqref{met1}, \eqref{met2}, \eqref{met3} and \eqref{l2c}, applying the substitution $y''=\frac{z'-y'}{\sqrt{\tau-s}}$, with $z^*=(z_1^*,z')$ and $y^*=(0,y')$,  we obtain
\[
\begin{aligned}
\int_{\Gamma}J_0(\mathcal L_l^{-1}(z^*),\tau;y,s)\d\sigma(y)=\sum_{j=0}^2P_j\left(\frac{z_1^*}{\sqrt{\tau-s}}\right) e^{-\frac{(z_1^*)^2}{\tau-s}}&\left[\int_{\R^{n-1}}\frac{J_0^j\left(z^*,\tau;y'',s;\tau-s\right)}{(\tau-s)^{\frac{j}{2}}}\d y''\right],
\end{aligned}
\]
for $0<s<\tau<t_0$ and $z^*\in \mathcal L_l(\Omega\cap W_{z_k})$, where, for $j=0,1,2$, $P_j$ are polynomials and  $J_0^j$ are continuous functions, $\mathcal C^1$ with respect to $\tau,\ s\in(0,t_0)$ and  satisfy
\[\max_{\substack{ i=0,1,2\\ \alpha_1+\alpha_2\leq1}}\int_{\R^{n-1}}\abs{\partial_\tau^{\alpha_1}\partial_s^{\alpha_2}J_0^j\left(z^*,\tau;y'',s;v_1\right)}\d y''\leq C_l,\quad 0<s<\tau<t_0,\ \ z_1^\ast >0,\ \ 0<\ v_1<t_0,\]
for some constant  $C_l>0$. We note that $\partial_{v_1}J_0^j\left((z_1',z''),\tau;y'',s; v_1\right)$ is not necessarily bounded. Indeed, we  show 
\[\abs{\partial_{v_1}J_0^j\left((z_1',z''),\tau;y'',s; v_1\right)}\leq \frac{C_l}{\sqrt{v_1}},\quad0<\ v_1<t_0,\ \  j=0,1,2.\]
This representation  and the construction of $K(z,\tau;y,s)$  in Chapter 5 of \cite{It} (see pages 31 to 32 for the construction in $\R^n$ and page 53 for the construction in a bounded domain) lead
\begin{equation}\label{l2d}
\int_{\Gamma}K(\mathcal L_l^{-1}(z^*),\tau;y,s)\d\sigma(y)= \sum_{j=0}^2 Q_j\left(\frac{z_1^*}{\sqrt{\tau-s}}\right)e^{-\frac{(z_1^*)^2}{\tau-s}}\left[\int_{\R^{n-1}}\frac{K_j\left(z^*,\tau;y'',s;\tau-s\right)}{(\tau-s)^{\frac{j}{2}}}\d y''\right],
\end{equation}
for $0<s<\tau<t_0$ and $z^*\in \mathcal L_l(\Omega\cap W_{z_k})$, where, for $j=0,1,2$, $Q_j$ are polynomials and  $K_j$ are continuous functions, $\mathcal C^1$ with respect to $\tau,\ s\in(0,t_0)$ and  satisfy
\[\
\max_{\substack{ i=0,1,2\\ \alpha_1+\alpha_2\leq1}}\int_{\R^{n-1}}\abs{\partial_\tau^{\alpha_1}\partial_s^{\alpha_2}K_j\left(z^*,\tau;y'',s;v_1\right)}\d y''\leq C_l,\quad 0<s<\tau<t_0,\ \ z_1^\ast >0,\ \ 0<\ v_1<t_0, 
\] 
where $C_l>0$ is a constant. Furthermore,  using representation \eqref{l2d}, we have, for $s<\tau<t<t_0$,  

\[\begin{aligned}&\int_\Omega H(x,t-\tau;z)\int_{\Gamma}K(\tau ,z;y,s)f(y,s)\d\sigma(y)\d z\\ &= \sum_{j=0}^2\sum_{l=1}^{m}\int_{\R_+^n}\omega_l(x)H_{z_l}(x, t-\tau ;\mathcal L_l^{-1}(z^*))\chi_l(z^*) Q_j\left(\frac{z_1^*}{\sqrt{\tau-s}}\right)e^{-\frac{(z_1^*)^2}{\tau-s}}\left[\int_{\R^{n-1}}\frac{K_j\left(z^*,\tau;y'',s;\tau-s\right)}{(\tau-s)^{\frac{j}{2}}}\d y''\right]\d z^* \end{aligned}\]
with $\R_+^n=\{(z_1^*,\ldots,z_n^*)\in\R^n;\; z_1^*>0\}$.
Then,  applying the substitutions $z''=\frac{x'-z'}{\sqrt{t-\tau}}$ and $z_1'=\frac{z_1^*}{\sqrt{\tau-s}}$, we deduce, in view of the form of the functions $K_j$,  the following
\begin{equation}\label{l2d1}
\begin{aligned}
&\int_\Omega H(x,t-\tau;z)\int_{\Gamma}K(z,\tau;y,s)\d\sigma(y)\d z\\ &=\sum_{j=0}^1\int_{\R_+^n}\frac{H_l'(x',t-\tau ;(z_1',z''), \tau-s)}{\sqrt{t-\tau}}\left[\int_{\R^{n-1}}\frac{K_j'\left((z_1',z''),\tau;y'',s;\tau-s\right)}{(\tau-s)^{\frac{j}{2}}}\d y''\right]\d z''\d z_1',\quad s<\tau<t<t_0,
\end{aligned}
\end{equation}
 for some continuous functions $K'_0$, $K'_1$ and $H_l'$ such that $K'_0$, $K'_1$ are $\mathcal C^1$, with respect to $s$ and $\tau$, and the following estimates hold:
\begin{equation}\label{rl1}
\int_{\R_+^n}\abs{H_l'(x',t-\tau ;(z_1',z''),\tau-s)} \d z''\leq C_l,\quad 0<s<\tau<t<t_0,
\end{equation}
\begin{equation}\label{rl2}
\max_{\substack{ j=0,1\\ \alpha_1+\alpha_2\leq1}}\int_{\R^{n-1}}\abs{\partial_\tau^{\alpha_1}\partial_s^{\alpha_2}K_j'\left((z_1',z''),\tau;y'',s; v_1\right)}\d y''\leq C_l,\quad 0<s<\tau<t_0,\ \ 0<\ v_1<t_0,
\end{equation}
for some constant  $C_l>0$.
  Repeating the arguments used for \eqref{l2d1} and applying some results of page 31 of \cite{It}, we obtain, for  $0<t<t_0$,
\begin{equation}\label{l2d2}\begin{aligned}\int_0^t\int_s^t\int_\Omega \abs{H(x,t-\tau ;z)}\int_{\Gamma}\abs{K(z,\tau;y,s)}\d\sigma(y) \d z  \d \tau  \d s&\leq C_l\int_0^t\int_s^t\left[\sum_{j=0}^1\frac{1}{\sqrt{t-\tau}}\cdot\frac{1}{(\tau-s)^{\frac{j}{2}}}\right]\d\tau \d s\\
\ &\leq C_l\sum_{j=0}^1\int_0^t(t-s)^{1-\frac{j}{2}}\d s\leq C_l.\end{aligned}\end{equation}
This estimate  and Fubini's theorem imply
\[F_2(x,t)=\int_0^t\int_s^t\int_\Omega H(x,t-\tau ;z)\int_{\Gamma}K(z,\tau;y,s)f(y,s)\d\sigma(y)\d z \d\tau \d s.\]
Then, in view of representation \eqref{l2d1}, for all $0<t<t_0$,   
\[
\begin{aligned}
F_2(x,t)=\int_0^t&\int_s^t\sum_{j=0}^1\int_{\R_+^n}\frac{H_l'(x',t-\tau ;(z_1',z''),\tau-s)}{\sqrt{t-\tau}}\\
\ &\times\left[\int_{\R^{n-1}}\frac{K_j'\left((z_1',z''),\tau ;y'',s;\tau-s\right)}{(\tau-s)^{\frac{j}{2}}}f_1(x',s;y'',z'')\d y''\right]\d z'' \d z_1' \d\tau \d s, 
\end{aligned}
\]
where $f_1(x',s;y'',z'')=f\left(\mathcal L_l^{-1}(0,x'-(\sqrt{t-s}) z''-(\sqrt{\tau-s}) y''),s\right)$. Making the substitution $\tau'=t-\tau$, we obtain
\[\begin{aligned}F_2(t,x)=\int_0^t&\int_0^{t-s}\sum_{j=0}^1\int_{\R_+^n}\frac{H_l'(x',\tau';(z_1',z''),t-s-\tau')}{\sqrt{\tau'}}\\
\ &\times\left[\int_{\R^{n-1}}\frac{K_j'\left((z_1',z''),t-\tau';y'',s;t-s-\tau'\right)}{(t-s-\tau')^{\frac{j}{2}}}f_1(x',s;y'',z'')\d y''\right]\d z''\d z_1'\d\tau'\d s.\end{aligned}\]
Then,  the substitution $s'=t-s$ yields
\[\begin{aligned}F_2(x,t)=\int_0^t&\int_0^{s'}\sum_{j=0}^1\int_{\R_+^n}\frac{H_l'(x',\tau';(z_1',z''),s'-\tau')}{\sqrt{\tau'}}\\
\ &\times\int_{\R^{n-1}}\frac{K_j'\left((z_1',z''),t-\tau';y'',t-s';s'-\tau'\right)}{(s'-\tau')^{\frac{j}{2}}}f_1(x',t-s';y'',z'')\d y''\d z_1'\d z''\d\tau'\d s'.\end{aligned}\]
But, for $0<\tau'<s'<t<t_0$, estimates \eqref{rl1}, \eqref{rl2} and  $f(y,0)=0$, $y\in\Gamma$, imply
\begin{equation}\label{rl3}
\begin{aligned}
&\Big| \sum_{j=0}^1\int_{\R_+^n}\frac{H_l'(x',\tau',x';(z_1',z''),s'-\tau')}{\sqrt{\tau'}}
\int_{\R^{n-1}}\frac{K_j'\left((z_1',z''),t-\tau';y'',t-s';s'-\tau'\right)}{(s'-\tau')^{\frac{j}{2}}}
\\
&\hskip 2 cm \times f_1(x',t-s';y'',z'')\d y'' \d z'' \d z_1'\Big|
\leq C_l\sum_{j=0}^1\frac{\norm{\partial_tf}_{L^\infty(\Gamma\times (0,t_0))}}{\sqrt{\tau'}} \frac{1}{(s'-\tau ')^{\frac{j}{2}}}
\end{aligned}
\end{equation}
and
\begin{equation}\label{rl4}\begin{aligned}&\Big|\partial_t\Big(\sum_{j=0}^1\int_{\R_+^n}\frac{H_l'(x',\tau ';(z_1',z''),s'-\tau')}{\sqrt{\tau'}}
\int_{\R^{n-1}}\frac{K_j'\left((z_1',z''),t-\tau ';y'',t-s';s'-\tau'\right)}{(s'-\tau ')^{\frac{j}{2}}}\\
 &\hskip 2 cm \times f_1(x',t-s';y'',z'')\d y'' \d z_1'\d z'' \Big)\Big|
\leq C_l\sum_{j=0}^1\frac{\norm{\partial_tf}_{L^\infty(\Gamma \times (0,t_0))}}{\sqrt{\tau '}}\frac{1}{(s'-\tau')^{\frac{j}{2}}}.\end{aligned}
\end{equation}
From estimates \eqref{rl3}, \eqref{rl4} and $f(y,0)=0$, $y\in\Gamma$, we conclude that $F_2$ admits a derivative with respect to $t$ 
and
\[
\begin{aligned}
\partial_tF_2(x,t)=\int_0^t\int_0^{s'}&\partial_t\Big( \sum_{j=0}^1\int_{\R_+^n}\frac{H_l'(x',\tau' ;(z_1',z''),s'-\tau')}{\sqrt{\tau'}}
\int_{\R^{n-1}}\frac{K_j' ((z_1',z''),t-\tau';y'',t-s';s'-\tau')}{(s'-\tau')^{\frac{j}{2}}}\\
&\times f_1(x',t-s';y'';z'')\d y'' \d z_1'\d z''\Big)\d\tau' \d s'.
\end{aligned}
\]
Moreover, \eqref{l2d2} and  \eqref{rl4} imply \eqref{l2gg} and \eqref{l2g}. Finally,  we obtain \eqref{l2a} from \eqref{l2dd} and \eqref{l2g}. This completes the proof.\qed

\section{Proof of Theorems 1 and 2}

{\it Proof of Theorem \ref{t1}.} Let $u=u_1-u_2$ and $\sigma =\sigma_2-\sigma_1$. Then $u$ is the solution of the following initial-boundary value problem
 
\begin{equation}
\label{eq3}\left\{
\begin{aligned}
&\partial_tu-\Delta_x u+\sigma_2(t)f(x)u=\sigma (t)f(x)u_1(x,t),\quad &(x,t)\in Q,
\\  
&u(x,0)=0, &x\in\Omega,
\\ 
&u(x,t)=0, &(x,t)\in\Sigma.
\end{aligned}
\right.
\end{equation}

 Let $U(x,t;y,s)$ be the fundamental solution of \eqref{eq2} with $q(x,t)=-\sigma_2(t)f(x)$. Applying Theorem 9.1 of \cite{It}, we obtain 
\begin{equation}\label{th4}
u(x,t)=\int_0^t\int_\Omega U(x,t;y,s)\sigma (s)f(y)u_1(y,s)\d y \d s +\int_0^t\int_\Gamma U(x,t;y,s)\partial_\nu u(y,s)\d\sigma( y )\d s.
\end{equation}
Now, since $u(x,t)=0$, $(t,x)\in\Sigma$ and $x\in\Gamma$, 
\begin{equation}\label{th5}\int_0^{t}\int_\Omega U(x,t;y,s)\sigma(s)f(y)u_1(y,s)\d y \d s =-\int_0^{t}\int_\Gamma U(x,t;y,s)\partial_\nu u(y,s)\d\sigma( y )\d s.\end{equation}

In view of differentiability properties in Lemma \ref{l1} and \ref{l2}, we can take the $t$-derivative of both sides of identity \eqref{th4}. We find

\[\begin{aligned}f(x)g(x,t)\sigma(t) =&-\int_0^{t}\partial_t\left(\int_\Omega U(x,t;y,s)\sigma(s)f(y)u_1(y,s)\d y\right) \d s\\ \ &-\partial_t\left(\int_0^{t}\int_\Gamma U(x,t;y,s)(\partial_\nu u_2(y,s)-\partial_\nu u_1(y,s))\d\sigma( y )\d s\right)\end{aligned}\]
and, for $x=x_0$, condition (H2)  implies
\begin{equation}\label{f1}
\begin{aligned}
\sigma (t)&=h(t)  \int_0^{t}\partial_t \left( \int_\Omega U(x_0,t;y,s)\sigma(s)f(y)u_1(y,s)\d y \right) \d s  
\\
&\hskip 3cm +h(t)\partial_t \left( \int_0^{t}\int_\Gamma U(x_0,t;y,s)\partial_\nu u(y,s)\d\sigma( y )\d s \right),
\end{aligned}
\end{equation}
 where $h(t)=-1/(g(t,x_0)f(x_0))$.

\smallskip 
Since $u(x,0)=0$, $x\in\Omega$, we have $\partial_\nu u(x,0)=0$, $x\in\Gamma$. Thus, the estimates in Lemma \ref{l1} and \ref{l2} lead
\begin{equation}\label{f2}\abs{\int_0^{t}\partial_t\left(\int_\Omega U(x_0,t;y,s)\sigma (s)f(y)u_1(y,s)\d y\right) \d s}\leq C\int_0^t\abs{\sigma (s)}\d s,\end{equation}
\begin{equation}\label{f3}\abs{\partial_t\left(\int_0^{t}\int_\Gamma U(x,t;y,s)\partial_\nu u(y,s)\d\sigma( y )\d s\right)}\leq C\norm{\partial_t\partial_\nu u}_{L^\infty(\Sigma)}.\end{equation}

Therefore, representation \eqref{f1}  and estimates \eqref{0}, \eqref{f2}, \eqref{f3} imply
 \[\abs{\sigma(t)}\leq\int_0^tC\abs{\sigma(s)}\d s+ C\norm{\partial_t\partial_\nu u}_{L^\infty(\Sigma)}.\]
Here and henceforth, $C>0$ is a generic constant depending only on data. Hence, Gronwall's lemma yields
\[\abs{\sigma(t)}\leq C\norm{\partial_t\partial_\nu u}_{L^\infty(\Sigma)}e^{Ct}\leq Ce^{CT}\norm{\partial_t\partial_\nu u}_{L^\infty(\Sigma)},\quad t\in(0,T).\]
Then \eqref{th1} follows and the proof is complete.
\qed

\smallskip
{\it Proof of Theorem 2.} Set $u=u_1-u_2=u(\sigma_1)-u(\sigma_2)$. Then, according to (H4) and (H6), $u$ is the solution of the following initial-boundary value problem
\begin{equation}\label{eqsem2}
\left\{
\begin{aligned}
&\partial_tu-\Delta_x u-q(x,t)u=F\left(t,x,\sigma_1(t),u_2(x,t)\right)-F\left(t,x,\sigma_2(t),u_2(x,t)\right),\quad &(x,t)\in Q,
\\  
&u(x,0) =0, &x\in\Omega,
\\ 
&u(t,x)=0, &(t,x)\in\Sigma ,
\end{aligned}
\right.
\end{equation}
with 
\begin{equation}\label{sem1}q(x,t)=\int_0^1\partial_uF\left[t,x,\sigma_1(t),u_2(x,t)+\tau(u_1(x,t)-u_2(x,t))\right]\d\tau .\end{equation}
Note that assumptions (H4) and (H6) imply that $q\in\mathcal C^1(\overline{Q})$.

\smallskip
On the other hand, in view of (H4), 
\[F\left(t,x,\sigma_1(t),u_2(x,t)\right)-F\left(t,x,\sigma_2(t),u_2(x,t)\right)=(\sigma_1(t)-\sigma_2(t))G(x,t),\]
with
\[G(x,t)=\int_0^1\partial_\sigma F\left(t,x,\sigma_2(t)+s(\sigma_1(t)-\sigma_2(t)),u_2(x,t)\right)\d s.\]
Using this representation, we deduce that $u$ is the solution of
\begin{equation}\label{eqsem3}
\left\{\begin{aligned}
&\partial_tu-\Delta_x u-q(x,t)u=(\sigma_1(t)-\sigma_2(t))G(x,t), \quad &(x,t)\in Q,
\\  
&u(x,0)=0, &x\in\Omega,
\\ 
&u(x,t)=0, &(t,x)\in\Sigma .
\end{aligned}\right.
\end{equation}
Let us remark that (H4) and (H6) imply that $G\in\mathcal C^{2,1}(\overline{Q})$. Let $U(t,x;s,y)$ be the fundamental solution of \eqref{eq2} with $q(x,t)$ defined by \eqref{sem1}. Then, according to Theorem 9.1 of \cite{It}, for $\sigma(t)=\sigma_1(t)-\sigma_2(t)$, we have the representation
\[u(x,t)=\int_0^t\int_\Omega U(x,t;y,s)\sigma(s)G(y,s)\d y \d s+\int_0^t\int_\Gamma U(x,t;y,s)\partial_\nu u(y,s)\d\sigma( y )\d s.\]
Since $u(x,t)=0$, $(t,x)\in\Sigma$, we obtain
\begin{equation}\label{sem2}\int_0^{t}\int_\Omega U(x_0,t;y,s)\sigma(s)G(y,s)\d y \d s =-\int_0^{t}\int_\Gamma U(x_0,t;y,s)\partial_\nu u(y,s)\d\sigma( y )\d s,\end{equation}
with $x_0$ defined in assumption (H5). Combining Lemma 1 and Lemma 2 with some arguments used in the proof of Theorem 1, we prove that $f_1$ and $f_2$ defined respectively by
\[f_1(t)=\int_0^{t}\int_\Omega U(x_0,t;y,s)\sigma(s)G(y,s)\d y d s,\]
\[f_2(t)=\int_0^{t}\int_\Gamma U(x_0,t;y,s)\partial_\nu u(y,s)\d\sigma( y )\d s,\]
admit a derivative with respect to $t$ and 
\[f_1'(t)=\sigma(t)G(x_0,t)+\int_0^{t}\partial_t\left(\int_\Omega U(x_0,t;y,s)\sigma(s)G(y,s)\d y \right)\d s,\]
\[f_2'(t)=\int_0^{t}\partial_t\left(\int_\Gamma U(x_0,t;y,s)\partial_\nu u(y,s)\d\sigma( y )\right)\d s,\]
\begin{align}\label{sem3}
&\abs{\int_0^{t}\partial_t\left(\int_\Omega U(x_0,t;y,s)\sigma(s)G(y,s)\d y \right)\d s}\nonumber
\\
& \hskip 4 cm \leq C\int_0^t\abs{\sigma(s)}\norm{G(\cdot,s)}_{\mathcal C^2_x(\overline{\Omega})}\d s\leq C \int_0^t\abs{\sigma(s)}\d s
\end{align}
and
\begin{equation}\label{sem4}
\abs{\int_0^{t}\partial_t\left(\int_\Gamma U(x_0,t;y,s)\partial_\nu u(y,s)\d\sigma( y )\right)\d s}
\leq C\norm{\partial_t\partial_\nu u}_{L^\infty(\Sigma)}.
\end{equation}
Here and in the sequel $C>0$ is a generic constant that can depend only on data.

\smallskip
Taking the $t$-derivative of both sides of identity \eqref{sem2}, we obtain
\begin{align*}
\sigma(t)G(x_0,t)&=-\int_0^{t}\partial_t\left(\int_\Omega U(x_0,t;y,s)\sigma(s)G(y,s)\d y \right)\d s
\\
&\hskip 3 cm -\int_0^{t}\partial_t\left(\int_\Gamma U(x_0,t;y,s)\partial_\nu u(y,s)\d\sigma( y )\right)\d s.
\end{align*}
Let us observe that (H5) and $\max (\norm{\sigma_1}_\infty, \norm{\sigma_2}_\infty)\leq M$ imply
\begin{align}\label{sem5} 
|G(x_0,t)|&=\int_0^1|\partial_\sigma F\left(t,x_0,\sigma_2(t)+s(\sigma_1(t)-\sigma_2(t)),g(x_0,t)\right)|\d s \nonumber
\\
&\geq \inf_{t\in[0,T],\sigma\in[-M,M]}|\partial_\sigma F(t,x_0,\sigma,g(x_0,t))|>0.
\end{align}
Then, 
\begin{align*}
\sigma(t)&=H(t)\int_0^{t}\partial_t\left(\int_\Omega U(x_0,t;y,s)\sigma(s)G(y,s)\d y \right)\d s
\\
&\hskip 2cm +H(t)\int_0^{t}\partial_t\left(\int_\Gamma U(x_0,t;y,s)\partial_\nu u(y,s)\d\sigma( y )\right)\d s ,
\end{align*}
where $H(t)=-1/G(x_0,t)$. Hence, \eqref{sem3}, \eqref{sem4} and \eqref{sem5} imply
\[\abs{\sigma(t)}\leq\int_0^tC\abs{\sigma(s)}\d s+ C\norm{\partial_t\partial_\nu u}_{L^\infty(\Sigma)}.\]
We complete the proof of Theorem \ref{t2} by applying Gronwall's lemma.
\qed


\begin{thebibliography}{99}
%


\bibitem[CE86-1]{CE86-1}{\sc 	J. R. Cannon and S. P. Esteva}, {\em An inverse problem for the heat equation}, Inverse Problems 2 (1986), 395-403.
\bibitem[CE86-2]{CE86-2}{\sc 	J. R. Cannon and S. P. Esteva}, {\em A note on an inverse problem related to the 3-D heat equation},
Inverse problems (Oberwolfach, 1986), 133-137, Internat. Schriftenreihe Numer. Math. 77,
Birkh\"auser, Basel, 1986.
\bibitem[CL88]{CL88}{\sc 	J. R. Cannon and Y. Lin}, {\em Determination of a parameter p(t) in some quasi-linear parabolic differential equations}, Inverse Problems  4 (1988), 35-45.
\bibitem[CL90]{CL90}{\sc 	J. R. Cannon and Y. Lin}, {\em An Inverse Problem of Finding a Parameter in a Semi-linear Heat Equation},  J. Math. Anal. Appl. 145 (1990), 470-484.m
\bibitem[Ch91-1]{Ch91-1}{\sc M. Choulli}, {\em An abstract inverse problem}, J. Appl. Math. Stoc. Ana. 4 (2) (1991) 117-128
\bibitem[Ch91-2]{Ch91-2}{\sc M. Choulli}, {\em An abstract inverse problem and application}, J. Math. Anal. Appl. 160 (1) (1991), 190-202.
\bibitem[Ch09]{Ch09}{\sc M. Choulli}, {\em Une introduction aux probl\`emes inverses elliptiques et paraboliques}, Math\'ematiques et Applications, 65, Springer-Verlag, Berlin, 2009.
\bibitem[COY]{COY}{\sc M. Choulli, E. M. Ouhabaz and M. Yamamoto},
{\em Stable determination of a semilinear term in a parabolic equation}, 
Commun. Pure Appl. Anal. 5 (3) (2006), 447-462.
\bibitem[CY06]{CY06}{\sc M. Choulli and  M. Yamamoto},
{\em Some stability estimates in determining sources and coefficients},
J. Inv. Ill-Posed Problems 14 (4) (2006), 355-373.
\bibitem[CY11]{CY11}{\sc M. Choulli and  M. Yamamoto}, {\em Global existence and stability for an inverse coefficient problem for a semilinear parabolic equation}, Arch. Math (Basel),  97 (6) (2011), 587-597.
\bibitem[E07]{E07}{\sc G. Eskin}, {\em Inverse hyperbolic problems with time-dependent coefficients}, Commun. PDE, 32  (11) (2007), 1737-1758.
\bibitem[E08]{E08}{\sc G. Eskin}, {\em Inverse problems for the Schr\"odinger equations with time-dependent electromagnetic potentials and the Aharonov-Bohm effect}, J. Math. Phys. 49 (2) (2008), 1-18.
\bibitem[It]{It}{\sc S. It\^o}, {\em Diffusion equations}, Transaction of Mathematical Monographs 114, Providence, RI, 1991.
\bibitem[LSU]{LSU}{\sc O. A. Ladyzhenskaja, V. A. Solonnikov and N. N. Ural'tzeva}, {\em Linear and
quasilinear equations of parabolic type}, Nauka, Moscow, 1967 in Russian ;
English translation : American Math. Soc., Providence, RI, 1968.
\bibitem[LS88]{LS88} {\sc A. Lorenzi and E. Sinestrari}, {\em An inverse problem in the theory of materials with
memory}, J.  Nonlinear Anal. TMA 12 (12) (1988),1217-1333.
\bibitem[LS87]{LS87}{\sc A. Lorenzi and E. Sinestrari}, {\em Stability results for a partial integrodifferential equation},
Proc. of the Meeting, Volerra Integrodifferential Equations in Banach Spaces; Trento, Pitman, London, (1987).
\bibitem[PO85-1]{PO85-1}{\sc  A. I. Prilepko and D. G. Orlovskii}, {\em Determination of evolution parameter of an equation,
and inverse problems in mathematical physics I}, Translations from Diff. Uravn. 21 (1) (1985), 119-129.
\bibitem[PO85-2]{PO85-2}{\sc  A. I. Prilepko and D. G. Orlovskii}, {\em Determination of evolution parameter of an equation,
and inverse problems in mathematical physics II},, Translations from Diff. Uranv.  21 (4) (1985), 694-701.
 \bibitem[Sa]{Sa}{\sc R. Salazar}, {\em Determination of time-dependent coefficients for a hyperbolic inverse problem}, arXiv:1009.4003v1.

\end{thebibliography}
\end{document}